\documentclass[13pt]{article}  
\usepackage{amssymb}              
\usepackage{amsthm}
\usepackage{amsmath}
\usepackage{eucal}
\usepackage{verbatim}
\usepackage[dvips]{graphicx}
\usepackage{multirow}
\usepackage{fancyhdr}
\usepackage{color}
\usepackage{enumerate}
\usepackage{calrsfs}
\usepackage{fullpage}
\usepackage{amsmath}
\usepackage{eufrak}

\usepackage{hyperref}

\usepackage[pagewise]{lineno} 

\newtheorem{theorem}{Theorem}
\newtheorem{lemma}{Lemma}

\newtheorem{remark}{Remark}

\makeatletter
\newcommand{\leqnomode}{\tagsleft@true}
\newcommand{\reqnomode}{\tagsleft@false}
\makeatother

\def\({\begin{eqnarray}}
\def\){\end{eqnarray}}
\def\[{\begin{eqnarray*}}
\def\]{\end{eqnarray*}}
\def\part#1#2{\frac{\partial #1}{\partial #2}}

\def\R{\mathbb{R}}
\def\N{\mathbb{N}}
\def\d{\mathrm{d}}
\def\tot#1#2{\frac{\d #1}{\d #2}}
\def\eps{\varepsilon}

\def\s{{\mathfrak{s}}}

\def\grad{\nabla}

\def\upsi{\underline\psi}
\def\I{\mathcal{I}}

\def\W{\mathcal{W}}
\def\P{\mathbb{P}}

\def\rev#1{{#1}}

\begin{document}

\title{Direct proof of unconditional asymptotic consensus in the Hegselmann-Krause model with transmission-type delay}   
\author{Jan Haskovec\footnote{Computer, Electrical and Mathematical Sciences \& Engineering, King Abdullah University of Science and Technology, 23955 Thuwal, KSA.
jan.haskovec@kaust.edu.sa}}

\date{}

\maketitle

\begin{abstract}
We present a direct proof of asymptotic consensus in the nonlinear Hegselmann-Krause model with transmission-type delay,
where the communication weights depend on the particle distance in phase space.
Our approach is based on an explicit estimate of the shrinkage of the group diameter on finite time intervals
and avoids the usage of Lyapunov-type functionals or results from nonnegative matrix theory.
It works for both the original formulation of the model with communication weights scaled by the number of agents,
and the modification with weights normalized a'la Motsch-Tadmor.
We pose only minimal assumptions on the model parameters.
In particular, we only assume global positivity of the influence function, without imposing any conditions on its decay rate or monotonicity.
Moreover, our result holds for any length of the delay.
\end{abstract}
\vspace{2mm}

\vspace{2mm}


\vspace{2mm}

\section{Introduction}\label{sec:Intro}
In this paper we study asymptotic behavior of the Hegselmann-Krause \cite{HK}
model with transmission-type time delay.
The Hegselmann-Krause model describes the evolution of $N\in\N$ agents who adapt their opinions to the ones of other members of the group.
Agent $i$'s opinion is represented by the quantity $x_i = x_i(t)\in\R^d$, with $d\in\N$ the space dimension, which is a function of time $t\geq 0$,
and evolves according to
\(   \label{eq:class}
   \dot x_i(t) = \sum_{j=1}^N \psi_{ij}(t) (x_j(t) - x_i(t)), \qquad i=1,\ldots,N.
\)
The communication weights $\psi_{ij}=\psi_{ij}(t)$ measure the intensity
of the influence between agents depending on the dissimilarity of their opinions.
In the classical setting \cite{HK} the communication weights are given by
\(  \label{psi:class}
    \psi_{ij}(t) = \frac{1}{N} \psi(|x_j(t) - x_i(t)|),
\)
where the nonnegative continuous \emph{influence function} $\psi:[0,\infty)\to [0,\infty)$,
also called \emph{communication rate},
measures how strongly each agent is influenced by others depending on their ``opinion distance".
Without loss of generality we may impose the global bound $\psi\leq 1$
(this can be always achieved by an eventual rescaling of time).

For many applications in biological and socio-economical systems \cite{Smith} or control problems (for instance, swarm robotics \cite{Hamman}),
it is natural to include a time delay in the model reflecting the time needed for each agent to receive information from other agents.
We therefore assume that agents' communication takes place subject to a time delay $\tau>0$,
i.e., agent $i$ with opinion $x_i(t)$ receives at time $t>0$ the information about the opinion of agent $j$
in the form $x_j(t-\tau)$.
The opinions evolve then according to the following variant of \eqref{eq:class},
\( \label{eq:HK}
   \dot x_i(t) = 
     \sum_{j\neq i} \psi_{ij}(t) (x_j(t-\tau) - x_i(t)), \qquad i=1,\ldots,N,
\)
where in the right-hand side we exclude the unjustified ``self-interaction" term $x_i(t-\tau) - x_i(t)$,
i.e., the summation runs over all $j\in \{1,2,\dots,N\} \setminus \{i\}$.
The system \eqref{eq:HK} is subject to the initial datum
\(\label{IC:HK}
   x_i(s) = x^0_i(s),\qquad i=1,\cdots,N, \quad s \in [-\tau,0],
\)
with prescribed trajectories $x^0_i\in C([-\tau,0])$, $ i=1,\cdots,N$.
Clearly, the presence of the delay shall be also reflected in the expression
for the communication weights $\psi_{ij}=\psi_{ij}(t)$.
We thus introduce the following modification of \eqref{psi:class},
\(   \label{psi}
   \psi_{ij}(t) :=  \frac{1}{N-1}   \psi(|x_j(t-\tau) - x_i(t)|).
\)
As pointed out in \cite{MT} for the second-order version of \eqref{eq:HK}, aka the Cucker-Smale system,
the scaling by $1/(N-1)$ in \eqref{psi} has the drawback that the dynamics of an agent is modified by the
total number of agents even if its dynamics is only significantly influenced by a few nearby
agents. Therefore, \cite{MT} proposed to normalize the communication weights relative to the influence
of all other agents, without involving explicit dependence on their number.
In the context of the delay system \eqref{eq:HK}, the normalized weights are given by
\(\label{psi:r0}
   \psi_{ij}(t) := 
   \frac{\psi(|x_j(t-\tau) - x_i(t)|)}{\sum_{\ell=1}^N \psi(|x_\ell(t-\tau) - x_i(t)|)}.
\)
However, since we excluded unjustified the ``self-interaction" term $x_i(t-\tau) - x_i(t)$ from the right-hand side of \eqref{eq:HK},
it seems appropriate to exclude the term $\psi(|x_i(t-\tau) - x_i(t)|)$ from the normalization as well, leading to
\(\label{psi:r1}
   \psi_{ij}(t) :=
        \frac{\psi(|x_j(t-\tau) - x_i(t)|)}{\sum_{\ell\neq i} \psi(|x_\ell(t-\tau) - x_i(t)|)}.
\)
However, our methods and results apply to any of the formulae \eqref{psi}--\eqref{psi:r1}.
In fact, the crucial property of the communication weights $\psi_{ij}=\psi_{ij}(t)$
that we impose for the rest of the paper is the following upper bound
\(  \label{psi:prop}
   \sum_{j\neq i} \psi_{ij}(t) \leq 1 \qquad\mbox{for all } i=1,\cdots, N \mbox{ and } t\geq 0.
\)
This fairly general assumption is obviously verified by the ``classical" weights \eqref{psi},
recalling that, without loss of generality, $\psi\leq 1$.
Clearly, it also holds for the normalized weights \eqref{psi:r0} and, with equality, for \eqref{psi:r1}.

The phenomenon of consensus finding in the context of \eqref{eq:HK} refers
to the (asymptotic) emergence of one or more {opinion clusters} formed
by agents with (almost) identical opinions \cite{JM}.
{Global consensus} is the state where all agents have the same opinion, i.e.,
$x_i=x_j$ for all $i,j \in\{1,\dots,N\}$.
\emph{Global asymptotic consensus} for the system \eqref{eq:HK} is then defined as the property
\( \label{def:consensus}
    \lim_{t\to\infty} |x_i(t) - x_j(t)| = 0 \qquad\mbox{for all } i,j\in \{1,\dots,N\}.
\)
The goal of this paper is to provide a proof of asymptotic global consensus
for the system \eqref{eq:HK} with an arbitrary delay length $\tau>0$ by explicitly estimating
the shrinkage of the group diameter in time.
The proof requires global positivity of the influence function,
\(  \label{psi:pos}
     \psi(s)>0 \qquad \mbox{for all } s\geq 0,
\)
which, for instance, includes the generic choice $\psi(s) = 1/(1+s^2)^\beta$ with $\beta\geq 0$
typically considered in consensus and flocking models.
Let us note that the global positivity property \eqref{psi:pos} is obviously
necessary for \emph{global} consensus to be reached in general.
Indeed, if $\psi$ was allowed to vanish even pointwise, i.e.,
$\psi(s)=0$ for some $s>0$, then one would have steady states
formed by two clusters of particles at distance $s$ apart.
Let us note that we do not require any monotonicity properties of $\psi$.


For the proof we take a direct approach, avoiding the usage of any Lyapunov-type functionals
or results from nonnegative matrix theory.
We first consider the spatially one-dimensional setting
and derive a uniform bound on the group diameter in terms of the initial datum.
Then, we obtain an explicit estimate on the diameter shrinkage
on finite time intervals. By an iterative argument we then conclude asymptotic convergence
of the group diameter to zero, i.e., consensus finding. Finally, we extend the result
to the multi-dimensional setting by observing that the method applies
for arbitrary 1D projections of the system.
The disadvantage of our method is that it does not provide convergence rates.
However, as the diameter shrinkage estimate is uniform with respect to the number of agents,
we are able to provide an extension of the consensus result to the mean-field limit
of the discrete system \eqref{eq:HK}.

Indeed, letting $N\to\infty$ in \eqref{eq:HK} leads to the conservation law
\(  \label{mf:HK}
   \partial_t f + \grad_x\cdot (F[f] f) = 0,
\)
for the time-dependent probability measure $f=f(t,x)$ which describes the probability
of finding an agent at time $t\geq 0$ located at $x\in\R^d$;
we refer to \cite{CCR} for details.
For the ``classical" communication weights \eqref{psi} the operator $F=F[f]$ is defined as
\(  \label{F}
   F[f](t,x) := \int_{\R^d} \psi(|x-y|) (y-x) f(t-\tau,y) \d y,   
\)
while for the normalized weights \eqref{psi:r0} or \eqref{psi:r1} we have
\(  \label{Fn}
   F[f](t,x) := \frac{\int_{\R^d} \psi(|x-y|) (y-x) f(t-\tau,y) \d y}{\int_{\R^d} \psi(|x-y|) f(t-\tau,y) \d y}.
\)
The system is equipped with the initial datum
$f(t) = f^0(t)$, $t\in [-\tau,0]$, with $f^0 \in C([-\tau,0], \P(\R^d))$,
where $\P(\R^d)$ denotes the set of probability measures on $\R^d$.
We shall assume that the initial datum is uniformly compactly supported,
i.e., 
\(   \label{IC:mf:HK}
   \sup_{s\in [-\tau,0]} d_x[f^0(s)] < \infty,
\)
where the diameter $d_x[h]$ for a probability measure $h\in \P(\R^d)$ is defined as
\(   \label{def:diam}
   d_x[h] := \sup \{ |x-y|,\, x,y\in \mathrm{supp}\, h \}.
\)

Convergence to global consensus as $t\to\infty$ for the system \eqref{eq:HK},
both with \eqref{psi} or \eqref{psi:r1}, has been proved in \cite{CPP} under a set of conditions requiring smallness of the maximal time delay
in relation to the fluctuation of the initial datum.
In contrast to this work, our paper offers a global consensus result under significantly weaker assumptions
- namely, for any time delay length, without restrictions on the decay of the influence function
(only assuming global positivity) and without smallness of the fluctuation of the initial datum.
In \cite{H} a simple proof of global consensus was given for the system \eqref{eq:HK}, however, exclusively with the
normalized communication weights \eqref{psi:r1}. This is because the method of proof is based
on a geometric argument, exploiting the convexity property of the weights \eqref{psi:r1},
namely, that  $\sum_{j\neq i} \psi_{ij} =1$ for all $i=1,\cdots, N$. Consequently,
it does not apply to the ``classical" case \eqref{psi}.

Models of consensus finding and flocking with delay have been extensively studied
in the Engineering community, see, e.g., \cite{E01,E2, Olfati-Saber},
typically based on tools from matrix theory, algebraic graph theory, and control theory.
For linear problems, where the communication weights $\psi_{ij}$ are fixed,
stability criteria based on the frequency approach and on Lyapunov-Krasovskii techniques
were derived in \cite{E3A}. We note that although our approach focuses on the nonlinear
setting with $\psi_{ij}$ nonlinearly depending on the agent configuration in phase space,
it of course could be easily modified to apply to the linear setting as a special case.
A simple proof of delay-independent consensus and flocking
in nonlinear networks with multiple time-varying delays under similarly mild assumptions as ours
was provided in \cite{E6}. This approach is based on fundamental concepts from the non-negative matrix theory
and a Lyapunov-Krasovskii functional. In contrast, our method estimates directly the
shrinkage of the group diameter, avoiding the use of Lyapunov-type functionals
and non-negative matrix theory.

The paper is organized as follows. In Section \ref{sec:mam} we formulate
our asymptotic consensus results for the discrete model \eqref{eq:HK}
and its mean-field limit \eqref{mf:HK}. In Section \ref{sec:proof} we provide the proof
for the discrete model, and in Section \ref{sec:cont} for the mean-field limit.

\section{Main results}\label{sec:mam}
Our main result regarding the global consensus behavior of the discrete
Hegselmann-Krause model with delay \eqref{eq:HK}--\eqref{IC:HK} is as follows.

\begin{theorem}\label{thm:HK}
Let the influence function $\psi\leq 1$ be continuous and strictly positive on $[0, \infty)$,
and let the weights $\psi_{ij}$ verify \eqref{psi:prop}.
Then all solutions of \eqref{eq:HK}--\eqref{IC:HK} reach global asymptotic consensus as defined by \eqref{def:consensus}.
\end{theorem}

An extension of the above consensus result for the mean-field limit \eqref{mf:HK}
is given by the following theorem,
which is a direct consequence of a stability estimate in terms of the
Monge-Kantorowich-Rubinstein distance, combined with the fact
that the consensus estimates derived in Section \ref{sec:proof} are uniform with respect
to the number of agents $N\in\N$.

\begin{theorem}\label{thm:mf:HK}
Let the influence function $\psi\leq 1$ be continuous and strictly positive on $[0, \infty)$.
Then all solutions $f=f(t)$ of \eqref{mf:HK} with $F=F[f]$ given either by \eqref{F} or \eqref{Fn},
subject to the compactly supported initial datum \eqref{IC:mf:HK},
reach global asymptotic consensus in the sense
\[  
   \lim_{t\to\infty} d_x[f(t)] = 0,
\]
with the diameter $d_x[\cdot]$ defined in \eqref{def:diam}.
\end{theorem}

\section{Proof of Theorem \ref{thm:HK}}\label{sec:proof}
  
The proof of Theorem \ref{thm:HK} is based on the following three Lemmata.
We first restrict to the spatially one-dimensional setting $d=1$, i.e., $x_i(t)\in\R$.
Having established the result for the 1D setting, it is easily generalized to the multi-dimensional
situation by considering arbitrary one-dimensional projections, as explained at the end of this Section.

We start by proving the following result
showing that the agent group remains uniformly bounded for all times,
the bound being given by the initial datum.

\begin{lemma}\label{lem:stay}
Let $d=1$ and denote
\(   \label{mM}
   m := \min_{i=1,\cdots,N} \min_{t\in [-\tau,0]} x_i(t),\qquad
   M := \max_{i=1,\cdots,N} \max_{t\in [-\tau,0]} x_i(t).
\)
Then, along the solutions of \eqref{eq:HK},
\(   \label{stay}
   m \leq x_i(t) \leq M
\)
for all $i\in\{1,\dots,N\}$ and all $t\geq 0$.
\end{lemma}

\begin{proof}
Let us fix any $i\in\{1,\dots,N\}$.
Using \eqref{eq:HK} and \eqref{mM}, we have for $t\in (0,\tau]$,
\[
   \dot x_i (t) &=& \sum_{j=1}^N \psi_{ij}(t) (x_j(t-\tau) - x_i(t)) \\
      &\leq&  \sum_{j=1}^N \psi_{ij}(t) (M - x_i(t)).
\]
Therefore, choosing any $\eps>0$, we obviously have
\(   \label{eq:lem:stay:1}
   \dot x_i (t) \leq \sum_{j=1}^N \psi_{ij}(t) \bigl((1+\eps)M - x_i(t)\bigr)
\)
for $t\in (0,\tau]$. Since $x_i(0) \leq M$, there exists $T\in (0,\tau]$ such that
$x_i(t) < (1+\eps)M$ for all $t\in [0,T)$. Assume, for contradiction, that $T<\tau$.
Then by continuity $x_i(T) = (1+\eps)M$.
But, using \eqref{psi:prop} in \eqref{eq:lem:stay:1} gives
\[
   \dot x_i (t) \leq (1+\eps)M - x_i(t)
\]
for $t\in [0,T)$, and integration gives
\[
   x_i(T) &\leq& x_i(0) e^{-T} + (1-e^{-T}) (1+\eps)M \\
      &\leq& M e^{-T} + (1-e^{-T}) (1+\eps)M < (1+\eps)M,
\]
a contradiction. Taking the limit $\eps\to 0$  we conclude that
$x_i(t) \leq M$ on the interval $[0,\tau]$.
The lower bound $x_i(t) \geq m$ on $[0,\tau]$ is obtained analogously.
Applying this argument inductively on consecutive time intervals of length $\tau$,
the claim follows.
\end{proof}

Having established the uniform bound \eqref{stay},
and observing that the system \eqref{eq:HK} is translation invariant,
we may without loss of generality adopt the assumption
\(  \label{ass:mM}
   0 < m \leq M
\)
in the sequel.
Lemma \ref{lem:stay} provides also a uniform bound on the particle speeds $|\dot x_i|$.
Indeed, with $x_i(t) \geq m > 0$ and \eqref{psi:prop} we have
\[
   \dot x_i (t) = \sum_{j\neq i} \psi_{ij}(t) (x_j(t-\tau) - x_i(t))
      \leq \sum_{j\neq i} \psi_{ij}(t) x_j(t-\tau)
      \leq  \sum_{j\neq i} \psi_{ij}(t) M
      \leq M,
\]
and
\[
   \dot x_i (t) = \sum_{j\neq i} \psi_{ij}(t) (x_j(t-\tau) - x_i(t))
      \geq  - \sum_{j\neq i} \psi_{ij}(t) x_i(t)
      \geq - \sum_{j\neq i} \psi_{ij}(t) M
      \geq - M,
\]
so that
\(  \label{s}
   |\dot x_i| \leq M \qquad \mbox{for all } t \geq 0 \mbox{ and } i\in\{1,\dots,N\}.
\)
Moreover, Lemma \ref{lem:stay} implies that, for all $t \geq 0$ and any $i, j\in \{1,\dots,N\}$,
\[
   |x_j(t-\tau)-x_i(t)| \leq |x_j(t-\tau)| + |x_i(t)| \leq 2M.
\]
Consequently,
\[
   \psi(|x_j(t-\tau)-x_i(t)|) \geq \min_{s\in [0,2M]} \psi(s).
\]
Let us denote
\(  \label{def:upsi}
    \upsi := \frac{1}{N-1} \min_{s\in [0,2M]} \psi(s),
\)
and note that $\upsi>0$ due to the assumption \eqref{psi:pos} of global positivity of $\psi$.
Then for both \eqref{psi} and \eqref{psi:r1} we obviously have (for \eqref{psi:r1} recall that $\psi\leq 1$),
\(  \label{upsi}
   \psi_{ij}(t) \geq \upsi \qquad \mbox{for all } t \geq 0 \mbox{ and } i, j\in\{1,\dots,N\}.
\)

In the sequel we shall use the following simple auxiliary lemma.

\begin{lemma} \label{lem:aux}
Let $t_0<t_1$ and $y\in C^1([t_0, t_1])$ satisfy, for some $\kappa\in\R$ and $\lambda>0$,
\rev{
\(  \label{aux:cond}
   \dot y(t) \leq \lambda(\kappa - y(t)) \quad
      \mbox{for all } t \in (t_0,t_1) \mbox{ such that } y(t) > \kappa.
\)
}
Then for all $t \in [t_0,t_1]$,
\(   \label{aux}
   y(t) \leq \rev{\max \left( \kappa, y(t_0) \right) } e^{-\lambda (t-t_0)} + \left(1 - e^{-\lambda (t-t_0)} \right) \kappa. 
\)
\end{lemma}

\begin{proof}
\rev{
We distinguish the following two cases:
\begin{itemize}
\item
If $y(t_0) \leq \kappa$, then by continuity we have $y(t) \leq \kappa$ for $t\in [t_0, T]$ for some $T\geq t_0$.
Let us choose $T$ to be maximal with this property.
If $T<t_1$, then $y(T)=\kappa$ and there exists $\eps>0$ such that $y(t) > \kappa$ for $t\in (T, T+\eps)$.
An integration of \eqref{aux:cond} on the interval $(T,t)$ gives
\[
       y(t) \leq y(T) e^{-\lambda(t-T)} + \left(1 - e^{-\lambda(t-T)} \right) \kappa = \kappa 
\] 
for $t\in (T, T+\eps)$, which is a contradiction to the choice of $T$.
We conclude by noting that for $y(t_0) \leq \kappa$ the claim \eqref{aux} reduces to $y(t) \leq \kappa$.
\item
If $y(t_0) > \kappa$, then, by continuity, there exists $T>t_0$ such that $y(t) > \kappa$ for $t\in [t_0,T)$.
Let us again choose $T$ to be maximal with this property.
Then, by \eqref{aux:cond} we have $\dot y(t) \leq \lambda(\kappa - y(t))$
for $t\in (t_0,T)$ and by integration we recover \eqref{aux} for $t\in [t_0,T]$.
If $T<t_1$, then $y(T) = \kappa$ and an application of the previous case gives
$y(t) \leq \kappa$ for $t\in [T, t_1]$, so that \eqref{aux} holds on the whole interval $[t_0, t_1]$.
\end{itemize}
}
\end{proof}

The following result is fundamental for our method of proof.
It shows that the spatial diameter of the agent group
shrinks on time intervals of length $6\tau$
by the explicit multiplicative factor $1-\Gamma$,
with
\(   \label{Gamma}
   \Gamma := \left(1 - e^{-\upsi\tau} \right)^2 (1-e^{-\sigma}) e^{-6\tau} \upsi,
\)
where $\upsi$ is given by \eqref{def:upsi} and
\(  \label{sigma}
   \sigma:=\min \left\{ \tau, \frac{M-m}{2M} \right\}.
\)

\begin{lemma}\label{lem:shrink}
Let $m$ and $M$ be given by \eqref{mM} and assume \eqref{ass:mM}.
Then, along the solutions of \eqref{eq:HK}, we have for all $i\in \{1,\dots,N\}$ and $t\in [5\tau,6\tau]$,
\(   \label{shrink:claim}
   m + \frac{\Gamma}{2} (M-m) \leq x_i(t) \leq M - \frac{\Gamma}{2} (M-m),
\)
with $\Gamma$ defined by \eqref{Gamma}.
\end{lemma}

\begin{proof}
The proof shall be carried out in two steps.

\textbf{Step 1.}
Due to 
\eqref{mM}, there exists an index $L\in\{1,\dots,N\}$ such that
$x_L(s) = m$ for some $s\in [-\tau,0]$.
The speed limit $|\dot x_i| \leq M$ given by \eqref{s}
then implies that there exists a closed interval $[\alpha_L, \omega_L] \subset [-\tau,0]$
of length $\sigma$ 
given by \eqref{sigma},
such that
\[  
   m \leq x_L(t) \leq \frac{m+M}{2} \qquad \mbox{for all } t \in [\alpha_L, \omega_L].
\]
Then we have for any $i\in\{1,\dots,N\} \setminus \{L\}$ and $t \in [\alpha_L + \tau, \omega_L + \tau]$,
\[
   \dot x_i (t) &=&
        \sum_{\substack{j\neq i \\ j\neq L}} \psi_{ij}(t) (x_j(t-\tau) - x_i(t)) +
        \psi_{iL}(t) (x_L(t-\tau) - x_i(t))  \\
      &\leq&
      \sum_{\substack{j\neq i \\ j\neq L}} \psi_{ij}(t) (M-x_i(t)) + 
      \psi_{iL}(t) \left( \frac{m+M}{2} - x_i(t) \right) \\
      &\leq&
      \bigl(1 - \psi_{iL}(t)\bigr) (M - x_i(t)) +  \psi_{iL}(t) \left( \frac{m+M}{2} - x_i(t) \right),
\]
where we used \eqref{psi:prop} in the last inequality.
With $\psi_{iL}(t) \geq \upsi$ due to \eqref{upsi}, we have
\[
   \bigl(1 - \psi_{iL}(t)\bigr) M + \psi_{iL}(t) \frac{m+M}{2} =
      M - \psi_{iL}(t) \frac{M-m}{2}
      \leq M - \upsi  \frac{M-m}{2}.
\]
Consequently,
\[
   \dot x_i (t) \leq  \left[ M - \upsi \frac{M-m}{2} \right] - x_i(t),
\]
for $t\in [\alpha_L+\tau, \omega_L+\tau]$.
Integration on the time interval $[\alpha_L+\tau, \omega_L+\tau]$,
recalling that $\sigma=\omega_L-\alpha_L$,
gives
\[
   x_i(\omega_L+\tau) &\leq& e^{-\sigma} x_i(\alpha_L+\tau) +  (1-e^{-\sigma}) \left[ M - \upsi \frac{M-m}{2} \right]   \\
     &\leq& e^{-\sigma}M +  (1-e^{-\sigma}) \left[ M - \upsi \frac{M-m}{2} \right]  \\
          &=& M -  \upsi (1-e^{-\sigma}) \frac{M-m}{2}.
\]
Inspecting the second line of the above formula, we note that the right-hand side is a convex combination of $M$ and $\left[ M - \upsi \frac{M-m}{2} \right]$.
Therefore, introducing the notation
\(  \label{gamma-}
   \gamma_- :=  \frac12 \upsi (1-e^{-\sigma})  \left(1 - \frac{m}{M} \right),
\)
we conveniently put the above estimate into the form
\(   \label{est:gminus}
   x_i(\omega_L+\tau) \leq (1-\gamma_-)M.
\)
Now, for $t\in[\omega_L+\tau, 6\tau]$ we have $x_j(t-\tau) \leq M$ by Lemma \ref{lem:stay},
and thus
\[
   \dot x_i (t) &=& \sum_{j\neq i} \psi_{ij}(t) (x_j(t-\tau) - x_i(t)) \\
      &\leq&  \sum_{j\neq i} \psi_{ij}(t) (M - x_i(t)) \\
      &\leq& M-x_i(t),
\]
where in the last estimate we used $x_i(t) \leq M$ and \eqref{psi:prop}.
Then, integration of the inequality $\dot x_i(t) \leq M-x_i(t)$
on the time interval $[\omega_L+\tau,t]$, with $t\in[\omega_L+\tau,6\tau]$, gives
\[
   x_i(t) \leq  e^{-(t-\omega_L-\tau)} x_i(\omega_L+\tau) + \left( 1 - e^{-(t-\omega_L-\tau)} \right) M,
\]
and \eqref{est:gminus} gives
\[
   x_i(t)
      &\leq&  e^{-(t-\omega_L-\tau)} \left(1 - \gamma_- \right) M + \left( 1 - e^{-(t-\omega_L-\tau)} \right) M  \\ 
      &=& \left( 1 - e^{-(t-\omega_L-\tau)}\gamma_- \right) M.  
\]
Since by definition $\omega_L \geq -\tau$, we have $t-\omega_L-\tau \leq t \leq 6\tau$, so that finally
\(
    x_i(t) \leq \left( 1 - e^{-6\tau}\gamma_- \right) M,  \label{upperB}
\)
for all $t\in[\omega_L+\tau,6\tau]$ and $i\in\{1,\dots,N\} \setminus \{L\}$.

Also, again due to \eqref{mM}, there exists an index $R\in\{1,\dots,N\}$ such that
$x_R(s) = M$ for some $s\in [-\tau,0]$.
Consequently, due to the speed limit \eqref{s},
\[  
   \frac{m+M}{2} \leq x_R(t) \leq M \qquad \mbox{for } t \in [\alpha_R, \omega_R],
\]
with an interval $[\alpha_R, \omega_R] \subset [-\tau,0]$ of length $\sigma$ given by \eqref{sigma}.
Following an analogous procedure as above, we derive the lower bound
\(   \label{lowerB}
   x_i(t) \geq \left( 1 + e^{-6\tau}\gamma_+ \right) m,
\)
for all $t\in[\omega_R+\tau,6\tau]$ and $i\in\{1,\dots,N\} \setminus \{R\}$, with
\(  \label{gamma+}
      \gamma_+ :=  \frac12 \upsi (1-e^{-\sigma}) \left(\frac{M}{m} - 1 \right).
\)

\textbf{Step 2.}
In the second step we derive estimates for $x_L$ and $x_R$, which have to be treated separately.
We note that we do not exclude the possibility that $L=R$, i.e., it may happen that $x_L$ and $x_R$ represent the same agent.
With \eqref{upperB}, we have for $t\in [2\tau,6\tau]$,
\[
   \dot x_L (t) &=& \sum_{j\neq L} \psi_{Lj}(t) (x_j(t-\tau) - x_L(t)) \\
      &\leq& \sum_{j\neq L} \psi_{Lj}(t) \left(\left( 1 - e^{-6\tau}\gamma_- \right) M - x_L(t)\right).
\]
Therefore, due to \eqref{upsi}, $x_L=x_L(t)$ on $[2\tau,6\tau]$ satisfies
\rev{
\[
   \dot x_L (t) \leq  
             \upsi  \left[ \left( 1 - e^{-6\tau}\gamma_- \right) M - x_L(t) \right] \qquad \mbox{if } x_L(t) > \left( 1 - e^{-6\tau}\gamma_- \right) M.
\]
Applying Lemma \ref{lem:aux} for $t\in [2\tau,6\tau]$, with $\lambda:=\upsi$ and $\kappa:=\left( 1 - e^{-6\tau}\gamma_- \right) M$, we obtain
\[
   x_L(t) &\leq& e^{-\upsi(t-2\tau)} \max\Bigl\{ x_L(2\tau), \left( 1 - e^{-6\tau}\gamma_- \right) M \Bigr\} + \left( 1 - e^{-6\tau}\gamma_- \right) M \left( 1 - e^{-\upsi(t-2\tau)} \right) \\
             &\leq& \left[ 1 - \left(1 - e^{-\upsi(t-2\tau)} \right) e^{-6\tau}\gamma_- \right] M.
\]
For the second inequality we used $x_L(2\tau)\leq M$, which gives
$\max\Bigl\{ x_L(2\tau), \left( 1 - e^{-6\tau}\gamma_- \right) M \Bigr\} \leq M$.
}
Restricting to $t \in [3\tau,6\tau]$, so that $t-2\tau \geq \tau$, we have
\(   \label{upperBL}
   x_L(t) \leq \bigl[ 1 - \left(1 - e^{-\upsi\tau} \right) e^{-6\tau}\gamma_- \bigr] M.
\)
Similarly, using \eqref{lowerB}, we derive the lower bound for $x_R$,
\(  \label{lowerBR}
   x_R(t) \geq \bigl[ 1 + \left(1 - e^{-\upsi\tau} \right) e^{-6\tau}\gamma_+ \bigr] m,
\)
for $t \in [3\tau,6\tau]$.

To derive a lower bound for $x_L$, we again write
\[
   \dot x_L (t) = \sum_{j\neq L} \psi_{Lj}(t) (x_j(t-\tau) - x_L(t)),
\]
this time for $t\in [4\tau, 6\tau]$.
To estimate the terms $x_j(t-\tau)$ on right-hand side, we use \eqref{lowerB} for $j\in\{1,\dots,N\} \setminus \{R\}$,
while for $j=R$ we use \eqref{lowerBR}; in the case $L=R$ we only need \eqref{lowerB}.
Since the right-hand side of \eqref{lowerBR} is smaller than the right-hand side of \eqref{lowerB},
we conveniently estimate
\rev{
\[
   \dot x_L (t) 
      \geq \sum_{j\neq L} \psi_{Lj}(t) \left( \bigl[ 1 + \left(1 - e^{-\upsi\tau} \right) e^{-6\tau}\gamma_+ \bigr] m - x_L(t)\right),
\]
and with \eqref{upsi},
\[
   \dot x_L (t) \geq    \upsi  \left( \bigl[ 1 + \left(1 - e^{-\upsi\tau} \right) e^{-6\tau}\gamma_+ \bigr] m - x_L(t) \right)
                \qquad    \mbox{if } x_L(t) < \bigl[ 1 + \left(1 - e^{-\upsi\tau} \right) e^{-6\tau}\gamma_+ \bigr] m,
\]
for $t\in [4\tau,6\tau]$.
An application of an obvious modification of Lemma \ref{lem:aux} on this time interval,
with $\lambda:=\upsi$ and $\kappa:=\bigl[ 1 + \left(1 - e^{-\upsi\tau} \right) e^{-6\tau}\gamma_+ \bigr] m$, gives
\[
   x_L(t) \geq e^{-\upsi(t-4\tau)} \min\Bigl\{x_L(4\tau), \bigl[ 1 + \left(1 - e^{-\upsi\tau} \right) e^{-6\tau}\gamma_+ \bigr] m \Bigr\}
        + \bigl[ 1 + \left(1 - e^{-\upsi\tau} \right) e^{-6\tau}\gamma_+ \bigr] m \left( 1 - e^{-\upsi(t-4\tau)} \right).
\]
With the estimate $x_L(4\tau)\geq m$ we have
\[
   \min\Bigl\{x_L(4\tau), \bigl[ 1 + \left(1 - e^{-\upsi\tau} \right) e^{-6\tau}\gamma_+ \bigr] m \Bigr\} \geq m,
\]
so that
\[
    x_L(t) \geq \left[ 1 + \left(1 - e^{-\upsi\tau} \right) \left( 1 - e^{-\upsi(t-4\tau)} \right)  e^{-6\tau}\gamma_+ \right] m,
\]
}
and restricting to $t\in [5\tau, 6\tau]$, so that $t-4\tau \geq \tau$,
\(   \label{lowerBL}
   x_L(t) \geq \left[ 1 + \left(1 - e^{-\upsi\tau} \right)^2 e^{-6\tau}\gamma_+ \right] m.
\)
Analogously, we derive
\(   \label{upperBR}
   x_R(t) \leq \bigl[ 1 - \left(1 - e^{-\upsi\tau} \right)^2 e^{-6\tau}\gamma_- \bigr] M,
\)
for $t\in [5\tau, 6\tau]$.

\medskip

We can finally summarize the estimates \eqref{upperB}, \eqref{lowerB} of Step 1
and \eqref{upperBL}--\eqref{upperBR} of Step 2 as
\[
    \left[ 1 + \left(1 - e^{-\upsi\tau} \right)^2 e^{-6\tau}\gamma_+ \right] m
        \leq x_i(t) \leq
             \left[ 1 - \left(1 - e^{-\upsi\tau} \right)^2 e^{-6\tau}\gamma_- \right] M,
\]
for $t\in [5\tau, 6\tau]$ and all $i\in \{1,\dots,N\}$.
Recalling the formulae \eqref{gamma-}, \eqref{gamma+} for $\gamma_-$, $\gamma_+$,
we obtain the claim \eqref{shrink:claim}.
\end{proof}

We are now in position to provide a proof of Theorem \ref{thm:HK}. Let us first consider the one-dimensional setting $d=1$.
We apply Lemma \ref{lem:shrink} iteratively on time intervals of total length $7\tau$.
Let us for $k\in\N$ denote $\I_k := [(6k-1)\tau,6k\tau]$ and
\[  
   m_k := \min_{i=1,\cdots,N} \min_{t\in \I_k} x_i(t),\qquad
   M_k := \max_{i=1,\cdots,N} \max_{t\in \I_k} x_i(t).
\]
Clearly, $m_0=m$ and $M_0=M$, with $m$ and $M$ given by \eqref{mM}.
Moreover, let us introduce the notation $D_k := M_k-m_k$ and
\[ 
   \Gamma_k := \left(1 - e^{-\upsi\tau} \right)^2 (1-e^{-\sigma_k}) e^{-6\tau} \upsi,
\]
with $\sigma_k := \min \left\{ \tau, \frac{M_k-m_k}{2M_k} \right\}$. Note that $\Gamma_k\in (0,1)$ as long as $M_k>m_k$.

An application of Lemma \ref{lem:shrink} gives
\[
    m_0 + \frac{\Gamma_0}{2} (M_0-m_0) \leq x_i(t) \leq M_0 - \frac{\Gamma_0}{2} (M_0-m_0) \qquad\mbox{for } t\in\I_1.
\]
We thus have
\[
   D_1 = M_1-m_1 \leq (1 - \Gamma_0) (M_0-m_0) = (1-\Gamma_0) D_0.
\]
Iterating Lemma \ref{lem:shrink} gives
\[
   D_{k+1} \leq (1-\Gamma_k) D_k 
      \qquad\mbox{for all } k\in\N.
\]
Due to Lemma \ref{lem:stay} we have $M_k\leq M$ for all $k\in\N$, so that
denoting $\widetilde\sigma(D) := \min \left\{ \tau, \frac{D}{2M} \right\}$,
we have $\sigma_k \geq \widetilde\sigma(D_k)$ for all $k\in\N$.
Then, with
\(   \label{wtgamma}
   \widetilde \Gamma(D) := \left(1 - e^{-\upsi\tau} \right)^2 (1-e^{-\widetilde\sigma(D)}) e^{-6\tau} \upsi,
\)
we have $\Gamma_k \geq \widetilde \Gamma(D_k)$ for all $k\in\N$ and
\(   \label{est:shrink}
   D_{k+1} \leq \left(1- \widetilde \Gamma(D_k)\right) D_k. 
\)
Clearly, the sequence $\{D_k\}_{k\in\N}$ is a nonnegative decreasing sequence, and denoting its limit $D$,
the limit passage $k\to\infty$ in the above inequality gives $ \widetilde\Gamma(D)\leq 0$,
which immediately implies $D=0$.
To conclude the proof of Theorem \ref{thm:HK} in the one-dimensional setting,
we apply the uniform bound of Lemma \ref{lem:stay}, i.e., if the agent group is
contained in the spatial interval $[m_k, M_k]$ for $t\in\I_k$ 
then it remains contained in the same interval for all future times.

Finally, generalization of the proof to the spatially multi-dimensional setting
is facilitated by the observation that the claims of both Lemma \ref{lem:stay}
and Lemma \ref{lem:shrink} can be trivially adapted to projections
of the trajectories $x_i=x_i(t)$ to arbitrary one-dimensional subspaces of $\R^d$.
Indeed, for an arbitrary fixed vector $\xi\in\R^d$, we replace \eqref{eq:HK} with the projected system
\[
   \tot{}{t} (x_i(t)\cdot \xi) = \sum_{j\neq i} \psi_{ij}(t) (x_j(t-\tau) - x_i(t))\cdot\xi, \qquad i=1,\ldots,N,
\]
and by a simple adaptation of the above proofs we obtain
\[
   \lim_{t\to\infty} (x_i(t)-x_j(t))\cdot\xi = 0.
\]
for all $i,j\in\{1,\dots,N\}.$ We conclude by choosing $\xi$'s as the basis vectors of $\R^d$.

\section{Proof of Theorem \ref{thm:mf:HK}}\label{sec:cont}
Our result for the mean-field limit system \eqref{mf:HK} with the operator $F=F[f]$ given either by \eqref{F} or \eqref{Fn}
is based on a straighforward modification of the well-posedness theory in measures developed in \cite[Section 3]{ChoiH1},
in particular, the existence, uniqueness  and continuous dependence
on the initial datum for measure-valued solutions of \eqref{mf:HK}.
The proof uses the framework developed in \cite{CCR}
and is based on local Lipschitz continuity of the operators \eqref{F} and \eqref{Fn}.
Let us provide here the stability result in Wasserstein distance \cite[Theorem 3.6]{ChoiH1},
which is essential for our proof of asymptotic consensus.

\begin{theorem}\label{thm:stabHK}
Let $f_1, f_2 \in C([0, T];\P(\R^{d}))$ be two measure-valued solutions of either \eqref{mf:HK}, \eqref{F}
or \eqref{mf:HK}, \eqref{Fn} on the time interval $[0, T]$,
subject to the compactly supported initial data $f_1^0, f_2^0 \in C([-\tau, 0];\P(\R^{d}))$.
Then there exists a constant $L=L(T)$ such that
\[ 
   \W_1 (f_1(t),f_2(t)) \leq L \max_{s\in[-\tau,0]} \W_1(f^0_1(s),f^0_2(s)) \quad \mbox{for} \quad t \in [0,T],
\]
where $\W_1(f_1(t),f_2(t))$ denotes the 1-Wasserstein (or Monge-Kantorovich-Rubinstein) distance \cite{Villani}
of the probability measures $f_1(t)$, $f_2(t)$. 
\end{theorem}

We are now in position to provide a proof of Theorem \ref{thm:mf:HK}.

\begin{proof}
Fixing an initial datum $f^0 \in C([-\tau,0], \P(\R^d))$, uniformly compactly supported in the sense of \eqref{IC:mf:HK},
we construct $\{f^0_N\}_{N\in\N}$ a family of $N$-particle approximations of $f^0$, i.e.,
\[
   f^0_N(s,x) := \frac{1}{N} \sum_{i=1}^N  \delta(x-x^0_i(s)) \qquad\mbox{for } s\in[-\tau,0],
\]
where the $x_i^0\in C([-\tau,0];\R^d)$ are chosen such that
\[
   \max_{s\in[-\tau,0]} \W_1(f^0_N(s),f^0(s)) \to 0 \quad\mbox{as}\quad N\to\infty.
\]
Denoting then $x^N_i=x^N_i(t)$ the solution of the discrete Hegselmann-Krause system \eqref{eq:HK}
subject to the initial datum $x_i^0=x_i^0(s)$, $i=1,\dots,N$,
it is easy to check that the empirical measure
\[
   f^N(t,x) := \frac{1}{N} \sum_{i=1}^N  \delta(x-x^N_i(t))
\]
is a measure valued solution of \eqref{eq:HK}, see \cite{CCR}.
The proof of Theorem \ref{thm:HK}
gives asymptotic convergence to global consensus, i.e., $d_x[f^N(t)] \to 0$ as $t\to\infty$,
with the diameter $d_x[\cdot]$ defined in \eqref{def:diam}.
Going back to the proof of Theorem \ref{thm:HK}, note that the
shrinkage estimate \eqref{wtgamma}--\eqref{est:shrink} does not
depend on the number of particles $N\in\N$.
Consequently, the convergence speed of $d_x[f^N(t)] \to 0$
depends on the initial datum $f^0_N$ only through the size of its support,
which is uniformly bounded due to \eqref{IC:mf:HK}.

For any fixed $T>0$, Theorem \ref{thm:stabHK} provides the stability estimate
\[
   \W_1 (f(t),f^N(t)) \leq L \max_{s\in[-\tau,0]} \W_1 (f^0(s),f^0_N(s)) \qquad \mbox{for} \quad t \in [0,T),
\]
where $f\in C([0, T];\P(\R^{d}))$ is a tight limit of $f^N$ as $N\to\infty$ and
the constant $L>0$ is independent of $N$.
Thus, fixing $T>0$ and letting $N\to\infty$ implies $d_x[f(t)] = \lim_{N\to \infty} d_x[f^N(t)]$ for $t\in [0,T)$, and, consequently,
\[  
   \lim_{t\to\infty} d_x[f(t)] = 0,
\]
\end{proof}


\begin{thebibliography}{99}

\bibitem{E01}
P.-A. Bliman, G. Ferrari-Trecate: \emph{Average consensus problems in networks of agents with delayed communications.}
Automatica 44 (2008), 1985--1995.

\bibitem{CCR}
\newblock {J. Ca\~nizo, J. Carrillo and J. Rosado},
\newblock {A well-posedness theory in measures for some kinetic models of collective motion},
\newblock  \emph{Math. Mod. Meth. Appl. Sci.}, \textbf{21} (2011), 515--539.

\bibitem{ChoiH1}
Y.-P. Choi and J. Haskovec: \emph{Cucker-Smale model with normalized communication weights and time delay.}
Kinetic and Related Models 10 (2017), 1011-1033.

\bibitem{CPP}
Y.-P. Choi, A. Paolucci and C. Pignotti:
\emph{Consensus of the Hegselmann-Krause opinion formation model with time delay}.
Math Meth Appl Sci. 2020; 1-- 20.

\bibitem{Hamman}
H. Hamman: \emph{Swarm Robotics: A Formal Approach}.
Springer, 2018.

\bibitem{H}
J. Haskovec: \emph{A simple proof of asymptotic consensus in the Hegselmann-Krause
and Cucker-Smale models with normalization and delay.}
SIAM J. on Applied Dynamical Systems, 20:1 (2021), 130--148.

\bibitem{HK}
R. Hegselmann and U. Krause, \emph{Opinion dynamics and bounded confidence models, analysis, and simulation},
J. Artif. Soc. Soc. Simul., 5, (2002), 1--24.

\bibitem{JM}
P.E. Jabin and S. Motsch: \emph{Clustering and asymptotic behavior in opinion formation.}
J. Differential Equations 257 (2014), 4165--4187.

\bibitem{MT}
S. Motsch and E. Tadmor: \emph{A New Model for Self-organized Dynamics and Its Flocking Behavior}.
J. Stat. Phys. 144 (2011).

\bibitem{E2}
S.-I. Niculescu: \emph{Delay Effects on Stability. A Robust Control Approach.}
Springer-Verlag London, 2001.

\bibitem{Olfati-Saber}
R. Olfati-Saber, J. A. Fax and R. M. Murray:
\emph{Consensus and Cooperation in Networked Multi-Agent Systems.}
Proceedings of the IEEE, vol. 95, no. 1 (2007), 215--233.

\bibitem{E3A}
A. Seuret, V. Dimos, V. Dimarogonas and K.H. Johansson:
\emph{Consensus under Communication Delays}.
Proceedings of the 47th IEEE Conference on Decision and Control,
Cancun, Mexico, Dec. 9-11, 2008.

\bibitem{Smith}
{H. Smith:}
\emph{An Introduction to Delay Differential
Equations with Applications to the Life Sciences.}
Springer New York Dordrecht Heidelberg London, 2011.


\bibitem{E6}
C. Somarakis and J. Baras:
\emph{Delay-independent convergence for linear consensus networks with applications to non-linear flocking systems}.
In Proceedings of the 12th IFAC Workshop on Time Delay Systems, pp. 159--164, Ann Arbor (2015).

\bibitem{Villani}
C. Villani: \emph{Topics in optimal transportation}.
Graduate Studies in Mathematics 58 (2003),
American Mathematical Society, Providence, RI.

\end{thebibliography}
\end{document}